\documentclass[12pt,twoside,reqno]{amsart}
\usepackage[utf8]{inputenc}
\usepackage{amsmath, amssymb, amsthm, enumitem, esint}
\usepackage{xcolor}
\usepackage[hyperindex]{hyperref}
\usepackage{graphicx}
\usepackage{tcolorbox}
\usepackage{appendix}

\usepackage{hyperref}
\hypersetup{bookmarksdepth=3}

\newcommand{\NN}{\mathcal{N}}

\renewcommand{\SS}{\mathcal{S}}

\newcommand{\IR}{\mathbb{R}}

\newcommand{\XX}{\mathcal{X}}

\DeclareMathOperator{\Ric}{Ric}
\DeclareMathOperator{\Rm}{Rm}

\DeclareMathOperator{\diam}{diam}

\newcommand{\EMPTY}[1]{}

\newtheorem{Theorem}[equation]{Theorem}
\newtheorem{Lemma}[equation]{Lemma}
\newtheorem{Corollary}[equation]{Corollary}
\newtheorem{Proposition}[equation]{Proposition}

\theoremstyle{definition}
\newtheorem{Definition}[equation]{Definition}
\theoremstyle{remark}

\numberwithin{equation}{section}

\def\XXint#1#2#3{{\setbox0=\hbox{$#1{#2#3}{\int}$ }
\vcenter{\hbox{$#2#3$ }}\kern-.6\wd0}}

\protected\def\vts{%
  \ifmmode
    \mskip0.5\thinmuskip
  \else
    \ifhmode
      \kern0.08334em
    \fi
  \fi
}

\begin{document}

\date{}

\author{Richard H $\text{Bamler}^*$}
\author{Bennett Chow}
\author{Yuxing $\text{Deng}^{**}$}
\author{Zilu Ma}
\author{Yongjia Zhang}



\thanks {* Supported by NSF grant DMS-1906500.}

\thanks {** Supported by NSFC grants 12022101 and 11971056.}

\address{ Richard H Bamler \\ Department of Mathematics,
University of California, Berkeley,
970 Evans Hall \#3840,
Berkeley, CA, 94720-3840 USA\\ rbamler@berkeley.edu.}

\address{ Bennett Chow\\ Department of Mathematics, University of California, San Diego, 9500 Gilman Drive \#0112, La Jolla, CA 92093-0112, USA\\ bechow@ucsd.edu.}

\address{ Yuxing Deng\\School of Mathematics and Statistics, Beijing Institute of Technology,
Beijing, 100081, China\\
6120180026@bit.edu.cn}

\address{ Zilu Ma\\ Department of Mathematics, University of California, San Diego, 9500 Gilman Drive \#0112, La Jolla, CA 92093-0112, USA\\ zim022@ucsd.edu.}

\address{ Yongjia Zhang \\School of Mathematics,
University of Minnesota, Minneapolis, MN 55455 \\
zhan7298@umn.edu
}


\title[4-Dimensional Steady Ricci Solitons with $3$-Cylindrical Tangent Flows at Infinity]{Four-Dimensional Steady Gradient Ricci Solitons with $3$-Cylindrical Tangent Flows at Infinity}

\maketitle

\textbf{Abstract.} In this paper we consider $4$-dimensional steady soliton singularity models, i.e., complete steady gradient Ricci solitons that arise as the rescaled limit of a finite time singular solution of the Ricci flow on a closed $4$-manifold.
In particular, we study the geometry at infinity of such Ricci solitons under the assumption that their tangent flow at infinity is the product of $\mathbb{R}$ with a $3$-dimensional spherical space form.
We also classify the tangent flows at infinity
of $4$-dimensional steady soliton singularity models in general. 

\section{Introduction}


Gaining a better understanding of the formation of singularities is one of the key goals in the study of higher-dimensional Ricci flows.
In this context, gradient solitons serve as an important class
of 
singularity models.
\emph{Steady} gradient solitons, in particular, are expected to play a crucial role in the study of Type II singularities (where the curvature blows up at
rate
$\gg (T-t)^{-1}$) and have been subject to ongoing research.
O.~Munteanu and J.~Wang \cite[Theorem 4.2]{MW11}
proved that any $n$-dimensional complete noncompact steady gradient Ricci soliton is either connected at infinity
(i.e., has exactly one end) or splits as the product of $\mathbb{R}$ with a compact Ricci
flat manifold.
In particular, any $4$-dimensional steady soliton singularity model must be connected at infinity.
In \cite{CFSZ20} it was shown that a $4$-dimensional steady soliton singularity model must also have bounded curvature.
In \cite{DZ20}, $4$-dimensional noncollapsed steady solitons with nonnegative sectional curvature decaying linearly are classified.
O.~Munteanu, C.-J.~Sung, and J.~Wang \cite{MSW19} proved that if a steady soliton has faster than linear curvature decay, then it must have exponential curvature decay.
In \cite{CLY11} it was shown that the curvature decay of a steady soliton is at most exponential under the assumption that the potential function $f$ satisfies $f(x)\to \infty$ as $x\to \infty$.
In \cite{MSW19} the assumption was weakened to $f$ being only bounded from below.


There has been a considerable amount of progress on \emph{shrinking} Ricci solitons.
For example, their asymptotic   behavior at infinity has been characterized by O.~Munteanu and J.~Wang \cite{MW15, MW17, MW19} in many settings.
Results regarding rigidity phenomena of shrinking solitons are due to T.~Colding and W.~Minicozzi \cite{CM21}, B.~Kotschwar and L.~Wang \cite{KW15, KW22}, and Y.~Li and B.~Wang \cite{LW21}.

In \cite{Bam20a,Bam20b,Bam20c}, the first author developed compactness and singularity theories in all dimensions.
In this paper, we apply these theories to certain questions regarding steady gradient Ricci solitons.
In particular, the main aim of this paper is to consider the case where the tangent flow at infinity is $3$-cylindrical.

Theorems 2.40 and 2.46 
in \cite{Bam20c}, stated for Ricci flows on closed manifolds, also hold for singularity models. 
Hence we have the following result of the first author (for the definition of the tangent flow at infinity, see \cite[\S 6.8]{Bam20b}). 
For the notation and definitions we use, see \S \ref{sec:Notation and prelims} below. 

\begin{Definition}
We say that a Ricci flow $(M, g(t)),t \in (-\infty,0)$ 
on a smooth orbifold with isolated singularities is 
a \emph{singularity model} if 
it is not isometric to Euclidean space and
it occurs as a blow-up model of a given Ricci flow $(\overline{M},\overline{g}(t)), t \in [0,T)$, $T<\infty$, on a compact manifold $\overline{M}$.
By this we mean that we can find a sequence of points $(x_i,t_i) \in \overline{M} \times [0,T)$ so that, after application of a time-shift by $-t_i$ and parabolic rescaling by some $\lambda_i \to \infty$, the metric flow pairs corresponding to  $(\overline{M}, \overline{g}(t), \nu_{x_i,t_i;t})$, $t \in [0,t_i)$, $\mathbb{F}$-converge to a metric flow pair $(\XX, (\nu_{x_\infty;t})), t< 0$, such  that 
$\XX$ 
is the metric flow induced by
$(M, g(t)),t \in (-\infty,0)$ (see \cite[\S 3.7]{Bam20b}). 

\end{Definition}
This notion is a generalization of the notion from \cite{CFSZ20}, as 
in particular
it also applies to the case 
we consider in this paper 
in which $M$ is a 4-dimensional smooth orbifold with isolated singularities and 
it does not require parabolic rescaling by the curvature at $(x_i,t_i)$.
For example, $\mathbb{R}^4/\Gamma$, where $\Gamma$ is a nontrivial subgroup of $O(4)$, is a candidate singularity model. 
The same can be said with $\mathbb{R}^4$ replaced by the Bryant soliton. 
This is the setting we will consider in this paper.
A more general notion of singularity model is considered in \cite{Bam20b}, \cite{Bam20c}, where it is proved that the singular set of a singularity model must have codimension $4$ in the parabolic sense.  

\begin{Theorem}\label{thm: tangent flow of 4d}
If $(M^4,g(t))$, $t\in (-\infty,0]$, is a $4$-dimensional singularity model on an orbifold with isolated singularities, then any tangent flow at infinity $(M^4_\infty,g_\infty(t))$, $t\in (-\infty,0)$, of $(M,g(t))$ is a $4$-dimensional, smooth, complete, shrinking gradient Ricci soliton on a Riemannian orbifold  with (isolated) conical singularities. Moreover, either $(M_\infty,g_\infty)$ is isometric to $\mathbb{R}^4/\Gamma$ for some nontrivial finite subgroup $\Gamma \subset O(4)$ or $R_{g_\infty(t)}>0$ on all of $M_\infty$.
For each $t<0$, the convergence to $(M_\infty,g_\infty (t))$ is in the smooth Cheeger--Gromov sense outisde of the discrete set of conical singularities.
\end{Theorem}

In this paper will prove the following result.

\begin{Theorem}
Let $(M^4,g,f)$ be a $4$-dimensional complete steady gradient Ricci soliton on an orbifold with isolated singularities  that is a singularity model.
Then the tangent flow at infinity is unique.
If the tangent flow at infinity is $(\mathbb{S}^3/\Gamma)\times \mathbb{R}$, then, for any $\epsilon>0$, outisde a compact set we have that each point is the center of an $\epsilon$-neck, has positive curvature operator, and linear curvature decay.

\end{Theorem}


Examples of $4$-dimensional steady solitons with tangent flows at infinity $(\mathbb{S}^3/\Gamma)\times \mathbb{R}$ are the Bryant soliton \cite{Bry05} and the Appleton \cite{App17} cohomogeneity one steady solitons on real plane bundles over $\mathbb{S}^2$.
On the other hand, examples of $4$-dimensional steady solitons with tangent flows at infinity $\mathbb{S}^2\times \mathbb{R}^2$ have been proven to exist by Yi Lai \cite{Lai20}.
In dimension $3$, she proved the existence of flying wing 
steady solitons 
as conjectured by Hamilton.

As pointed out by the first author in \cite[Section 2.7]{Bam20c}, the
tangent flows at infinity should agree with Perelman's asymptotic solitons constructed in \cite[Section 11]{Per02}.
This was recently confirmed in \cite{CMZ21a} by P.-Y.~Chan and two of the authors. 
In \cite{MZ21} by two of the authors,
Perelman's constructions are studied on complete ancient Ricci flows, rather than only singularity models, with different curvature conditions from those in previous approaches.

\textbf{Acknowledgment:} We would like to thank the referee for a number of suggestions which improved the paper. 

\section{Notation and preliminaries}\label{sec:Notation and prelims}

For background on orbifolds, see Chapter 13 of Thurston's book \cite{T21}. In this paper we consider smooth $4$-dimensional orbifolds $M$ with isolated singularities, so that the local model at a singular point is $\mathbb{R}^4/\Gamma$, where $\Gamma$ is a finite subgroup of $O(4)$.
A Riemannian metric $g$ on $M$ is smooth when the local lifts to $\mathbb{R}^4$ are smooth.

We say that $(M,g,f)$ is a steady gradient Ricci soliton if $\Ric = \nabla^2 f$, and is a shrinking gradient Ricci soliton if $\Ric = \nabla^2 f + \frac{1}{2}g$; see \cite{Ham93b}. By passing to the local lifts, these equations hold on all of $M$, not just its regular part. In particular, $\nabla f=0$ at each
isolated
singular point. 
So the flow of $-\nabla f$ can be defined by passing to local lifts and it preserves the set of regular points.

By an ALE space, we mean an asymptotically locally Euclidean space; see S.~Bando, A.~Kasue, and H.~Nakajima \cite{BKN89}.

For a Ricci flow,
the notions and properties of heat kernel $\nu_{x_0,t_0;s_0}$, $H_n$-center, and 
pointed Nash entropy $\mathcal{N}_{x,t}(\tau)$ 
are defined in \cite{Bam20a}.
Defined by the first author are the notions and properties of 
metric flow (generalizing Ricci flow) and
metric soliton (generalizing gradient Ricci soliton) \cite[\S 3]{Bam20b},
$\mathbb{F}$-distance \cite[\S 5]{Bam20b}, and $\mathbb{F}$-convergence (generalizing Cheeger--Gromov convergence), $\mathbb{F}$-limit, and tangent flow at infinity \cite[\S 6]{Bam20b}.

Throughout this paper, unless otherwise specified, we will be in the category of smooth $4$-dimensional orbifolds with
a finite number of
isolated singularities.

\section{Proofs}


In view of Theorem \ref{thm: tangent flow of 4d}, via a splitting result and the classification of $3$-dimensional shrinking solitons, we may classify the possible tangent flows at infinity of $4$-dimensional steady soliton singularity models.
As indicated earlier, we will be in the category of smooth orbifolds with isolated singularities.

\begin{Proposition}\label{prop: tangent flows posibilities}
Any tangent flow at infinity $(M^4_\infty,g_\infty(t))$, $t\in (-\infty,0)$, of a nontrivial $4$-dimensional steady gradient Ricci soliton singularity model $(M^4,g(t))$, $t\in (-\infty,0]$, is either $\mathbb{R}^4/\Gamma$ (but not $\mathbb{R}^4$), $(\mathbb{S}^3/\Gamma)\times \mathbb{R}$, $\mathbb{S}^2 \times \mathbb{R}^2$, or $((\mathbb{S}^2 \times \mathbb{R})/\mathbb{Z}_2) \times \mathbb{R}$. If the tangent flow at infinity is $\mathbb{R}^4/\Gamma$, then $(M,g(t))$ is a (static) Ricci-flat ALE space.
\end{Proposition}

\begin{proof}
Firstly, we remark that the definition of a tangent flow at infinity, which uses a 
space-time basepoint $(x_0,t_0) \in M\times (-\infty,0]$
and a sequence $\lambda_i \to 0$, may depend on $\lambda_i$ but is independent of the choice of  
$(x_0,t_0)$;
see \cite[Definition 6.55]{Bam20b} and 
\cite[Theorem 1.6]{CMZ21b}.
By \cite[Theorem 6.58]{Bam20b}, any tangent flow at infinity of a finite time singularity model can be realized as an $\mathbb{F}$-limit of a sequence of compact Ricci flows (rescalings of the original Ricci flow). 
By \cite[\S 2.7]{Bam20c},
the Nash entropy of the sequence is uniformly bounded away from $-\infty$ and thus the tangent flows at infinity of singularity models always exist (even if they do not have bounded curvature).

We claim that each tangent flow at infinity is either $\mathbb{R}^4/\Gamma$ ($\Gamma \neq 1$ by 
\cite[Theorem 2.40]{Bam20c})
or splits off a line. In the latter case, since it is a smooth orbifold with conical singularities, by Theorem \ref{thm: tangent flow of 4d} it must be the product of $\mathbb{R}$
with a complete shrinking gradient Ricci soliton (not necessarily with bounded curvature) on a $3$-dimensional smooth manifold with $R>0$. The proposition now follows since these have been classified as $\mathbb{S}^3/\Gamma$, $\mathbb{S}^2 \times \mathbb{R}$ and $(\mathbb{S}^2 \times \mathbb{R})/\mathbb{Z}_2$; see Hamilton \cite[\S 26]{Ham93b}, Perelman \cite[Lemma 1.2]{Per03}, Cao, Chen, and Zhu \cite{CCZ08}, Ni and Wallach \cite{NW08}, and Petersen and Wylie \cite{PW10}.

Now $\mathbb{F}$-convergence (see \cite[Definition 6.2]{Bam20b}, when the limit is an orbifold with conical singularities, can be upgraded to pointed Cheeger--Gromov convergence with respect to $H_n$-centers\footnote{For the definition of $H_n$-center, see \cite[Definition 3.10]{Bam20a}.} smoothly on compact subsets of the limit minus the conical singularities; see \cite[\S 9.4]{Bam20b}.

To prove the claim in the first paragraph of this proof, we consider two cases: (1) $\nabla f$ remains locally bounded and (2) $\nabla f$ goes to infinity.
Suppose that the rescalings $(M, \lambda_i g (-\lambda_i^{-1}), 
z_i)$ of a steady soliton model, where $(z_i,-\lambda_i^{-1})$ is an $H_n$-center of $(x_0,0)$ and $\lambda_i \to 0$, limit to a complete shrinking gradient Ricci soliton $(M_\infty,g_\infty,z_\infty)$ on a $4$-orbifold with conical singularities, after pulling back by diffeomorphisms $\phi_i$.
Let $\SS_\infty$ denote the set of conical singularities of $M_\infty$, which is a discrete set of points, and let $\mathcal{R}_\infty = M_\infty-\SS_\infty$.

We may assume that the steady soliton solution to the Ricci flow $g(t)$ is equal to $\Phi_t^*g$, where 
$\Phi_t$ is the $1$-parameter group of diffeomorphisms generated by $-\nabla_g f$.
We define $f(x,t)=f(\Phi_t(x))$,
so that $\Ric_{g(t)} = \nabla^2_{g(t)} f(t) $.
Let $g_i = \lambda_i g (-\lambda_i^{-1})$ and $f_i :=f(\cdot,-\lambda_i^{-1})$.
We have $z_\infty \in \mathcal{R}_\infty$ and we have smooth pointed Cheeger--Gromov convergence of $(M,g_i,z_i)$ to the limit on compact subsets of $\mathcal{R}_\infty$ (see \cite[\S 9]{Bam20b}).

\smallskip

\emph{Case 1}:
\emph{Suppose that, for a subsequence, $|df_i |_{g_i}(z_i)$ is uniformly bounded.}
Pass to this subsequence.
Let $\bar{f}_i = f_i - f_i (z_i)$. From the smooth convergence, we have that $|{\Rm_{g_i}}|$ is uniformly bounded away from the conical singularities of the limit (after pulling back by the diffeomorphisms $\phi_i$).
In particular, the consequent Ricci curvature bound and the steady soliton equation imply that $|\nabla^2_{g_i}\bar{f}_i|_{g_i} \leq C$ on compact subsets of $\mathcal{R}_\infty$.
Since $\mathcal{R}_\infty$ is connected and $|d\bar{f}_i|_{g_i}(z_i)\leq C$, this implies that $|d\bar{f}_i|_{g_i} \leq C$ on compact subsets of $\mathcal{R}_\infty$.

Thus, by 
$\bar{f}_i(z_i)=0$, $|\nabla \bar{f}_i|_{g_i} \leq C_1(d(\cdot,z_i))$, and
Shi's local derivative of curvature estimates, we have that $|\nabla^k \bar{f}_i|_{g_i} \leq C_k(d(\cdot,z_i))$ for all $k \geq 0$
on compact subsets of $\mathcal{R}_\infty$. 
Hence the $\bar{f}_i$ subconverge to a smooth function $f_\infty$ on $\mathcal{R}_\infty$.
By taking the limit of the steady soliton equation $\Ric_{g_i} = \nabla^2_{g_i} \bar{f}_i$, we obtain
$\Ric_{g_\infty} = \nabla^2_{g_\infty} f_\infty$ on $M_\infty$ minus the conical singularities. 
On the other hand, since $(M_\infty,g_\infty)$ has a shrinking gradient Ricci soliton structure, there exists a function $f_0$ such that $\Ric_{g_\infty} = \nabla^2_{g_\infty} f_0 + \frac{1}{2}g_\infty$, so that $h:=f_\infty - f_0$ satisfies $\mathcal{L}_{\nabla h} g_\infty = 2\nabla^2_{g_\infty} h = g_\infty$ on $M_\infty$ minus the conical singularities. 
By adjusting $h$ by an additive constant if necessary, this implies that $|\nabla h|_{g_\infty}^2 = \frac{1}{2}h$.
Hence $\rho := 2\sqrt{h}$ satisfies $|\nabla \rho|_{g_\infty}\equiv 1$
and $\nabla_{\nabla \rho} \nabla \rho \equiv 0$ on $\mathcal{R}_\infty$, so that the integral curves of $\nabla \rho$ are unit speed geodesics.
This implies that $(M_\infty, g_\infty)$ is a flat cone whose cross sections are the level sets of $h$.

Since the conical singularities are orbifold points, this implies that $(M_\infty, g_\infty)= \mathbb{R}^4/\Gamma$, where $\Gamma$ is a finite subgroup of $O(4)$.
Therefore, on $(M,g)$, we have $R_g (w_i) =\lambda_i R_{g_i}(z_i) \to 0$,
where $w_i=\Phi_{-1/\lambda_i}(z_i).$
We also have that $|df|^2_g (w_i) = \lambda_i |d f_i|^2_{g_i} (z_i) \to 0$. So on $(M,g)$, $R + |df|^2 = C = 0$, which implies $\Ric_g = 0$. 
Since the steady soliton singularity model has $R_g \equiv 0$, 
by 
the first author's
generalization of Perelman's no local collapsing theorem \cite[Theorem 6.1]{Bam20a},
there exists $\kappa > 0$ such that $\operatorname{Vol}_g(B^g_r (x_0)) \geq \kappa r^4$ for $r > 0$; 
hence, by definition, $g$ has Euclidean volume growth.
It now follows 
 from Cheeger and Naber \cite[Corollary 8.85]{ChN15} 
that $(M,g)$
is an ALE space.
Note that $\Gamma \neq 1$ also follows from the equality case of the Bishop--Gromov volume comparison theorem.
\smallskip 

\emph{Case 2: Suppose that, for a subsequence, $|df_i|_{g_i}(z_i) := \beta_i^{-1} \to \infty$.}
Pass to this subsequence.
Let $\bar{f}_i := \beta_i (f_i-f_i(z_i))$.
Then $\bar{f}_i(z_i)=0$, $|d\bar{f}_i|_{g_i}(z_i)=1$, and $\nabla^2_{g_i} \bar{f}_i \to 0$ on compact subsets of $\mathcal{R}_\infty$. Again, we have higher derivative estimates for $\bar{f}_i$.
Thus, the $\bar{f}_i$ subconverge to a smooth function $f_\infty$ on $\mathcal{R}_\infty$ satisfying $\nabla^2_{g_\infty} f_\infty = 0$ on $\mathcal{R}_\infty$ and $|df_\infty|_{g_\infty}(z_\infty)=1$. This implies the splitting of $(\mathcal{R}_\infty,g_\infty)$. Since the singularities are conical, there are no singularities and hence $(M_\infty,g_\infty)$ splits.
\end{proof}

The discreteness of the space of $3$-dimensional shrinking solitons occurring in Proposition \ref{prop: tangent flows posibilities} implies the following.

\begin{Proposition}\label{lem: uniqueness of tangent flow}
Any $4$-dimensional steady gradient Ricci soliton singularity model $(M^4,g(t))$, with potential function $f(t)$, has a unique
tangent flow at infinity.
\end{Proposition}

\begin{proof}
If one tangent flow at infinity is $\IR^4/\Gamma$, then $(M,g(t))$ is a Ricci flat ALE space as we have seen in the proof of Proposition \ref{prop: tangent flows posibilities}, and thus in this case any tangent flow at infinity is $\IR^4/\Gamma$.
So we may assume that no tangent flow at infinity is $\IR^4/\Gamma.$

Let $\mathcal{X}$ be the metric flow induced by the Ricci flow $(M^4,g(t))$; see \cite[Definition 3.2]{Bam20b}.
Let $I=[-2,-1/2]$ and let
\[
    \mathcal{T} 
    := \big\{
    \text{metric solitons $(\mathcal{Y}, (\mu_t))$ that arise as 
    tangent flows at infinity of }\mathcal{X},
    \text{ restricted to } I
    \big\}\, ;
\]
see \cite[Definition 3.57]{Bam20b} for the definition of metric soliton,
and see \cite[Definition 3.10]{Bam20b} for the 
definition of the restriction of a metric flow.
By Proposition \ref{prop: tangent flows posibilities}, the elements of $\mathcal{T}$ are the metric solitons  associated to 
$N\times \mathbb{R}$, where $N$ is a $3$-dimensional complete shrinking gradient Ricci soliton structure that is isometric to $\mathbb{S}^3/\Gamma$, $\mathbb{S}^2 \times \mathbb{R}$, or $(\mathbb{S}^2 \times \mathbb{R})/\mathbb{Z}_2$.
	Note that these are the splitting quotients of $\mathbb{S}^k \times \mathbb{R}^{4-k}$, with the metrics $2(k-1)g_{\mathbb{S}^k}+g_{\mathbb{R}^{4-k}}$, $k=2,3$.
	Hence the metric space $(\mathcal{T},d_{\mathbb{F}}^J)$ 
	is discrete, where $d_{\mathbb{F}}^J$ denotes the $\mathbb{F}$-distance
	introduced in \cite[\S 5.1]{Bam20b} and where 
	$J$ is taken to be 
	$\{-1\}$
	for convenience.
	By \cite[Theorem 7.4]{Bam20b}, $\mathcal{T}$ is compact and thus finite.

\def \X {\mathcal{X}}
\def \Y {\mathcal{Y}}
	
Let $10 \epsilon$ be the smallest distance between elements of $(\mathcal{T},d_{\mathbb{F}}^J)$ and suppose that this distance is attained by
$(\Y^k_I,\mu^k_t)\in \mathcal{T},k=0,1$, i.e., 
\begin{equation*}
	10 \epsilon = d_{\mathbb{F}}^J
	\left((\Y^0_I,(\mu^0_t)),(\Y^1_I,(\mu_t^1))
	\right).
\end{equation*}
Then there are sequences of scales $ \lambda_{k,j}\to 0$ as $j\to \infty$ such that
\begin{equation*}
	\lim_{j\to \infty} 
	d_{\mathbb{F}}^J
	\left((\Y^k_I,(\mu_t^k)),
	\left(\X^{0, \lambda_{k,j}}_I, \big(\nu^{0,\lambda_{k,j}}_{x_0;t}\big)
	\right)
	\right) \to 0,
\end{equation*}
for $k=0,1 $ and where $\X^{-\Delta T, \lambda}$ denotes the time-shift by 
$-\Delta T$ and then parabolic rescaling by $\lambda$ of $\X$ as in 
\cite[\S 6.8]{Bam20b}.

By discarding some scales, we may assume that $ \lambda_{0,j}<\lambda_{1,j}$.
There is a $\bar j$ such that if $j\ge \bar j$,
\begin{equation*}
	d_{\mathbb{F}}^J
	\left((\Y^k_I,(\mu_t^k)),
	\left(\X^{0, \lambda_{k,j}}_I, \big(\nu^{0,\lambda_{k,j}}_{x_0;t}\big)
	\right)
	\right)
	< \epsilon.
\end{equation*}
It follows that
\begin{equation*}
	d_{\mathbb{F}}^J
	\left(
	\big(\X^{0, \lambda_{0,j}}_I, \big(\nu^{0,\lambda_{0,j}}_{x_0;t}\big)
	\big),
	\big(\X^{0, \lambda_{1,j}}_I, \big(\nu^{0,\lambda_{1,j}}_{x_0;t}\big)
	\big)
	\right) > 8 \epsilon.
\end{equation*}
Note that there is a continuous curve connecting the two rescaled flows:
\begin{equation*}
	\gamma_j(\eta) = 
	\big(\X^{0, \eta}_I, \big(\nu^{0,\eta}_{x_0;t}\big)
	\big)
\end{equation*}
for $\eta\in [ \lambda_{0,j}, \lambda_{1,j}].$  
So there is some 
$\eta_j\in ( \lambda_{0,j}, \lambda_{1,j})$  such that
\begin{equation*}
	d_{\mathbb{F}}^J\left( \gamma_j(\eta_j), 
	\big(\X^{0, \lambda_{0,j}}_I, \big(\nu^{0,\lambda_{0,j}}_{x_0;t}
	\big)\big)\right)
	\in [ 2 \epsilon, 4 \epsilon] \, ;
\end{equation*}
meanwhile,
\begin{equation*}
	d_{\mathbb{F}}^J
	\left( \gamma_j(\eta_j), 
	\big(\X^{0, \lambda_{1,j}}_I, \big(\nu^{0,\lambda_{1,j}}_{x_0;t}
	\big)\big)\right)
	> 2 \epsilon.
\end{equation*}

By the existence of tangent flows at infinity, 
a subsequence of $\gamma_j(\eta_j)$ converges to a splitting metric soliton
$(\mathcal{Z},(\mu_t))$. Hence
\begin{equation*}
	d_{\mathbb{F}}^J
	\left( (\mathcal{Z}_I,(\mu_t)),
	(\Y^0_I,(\mu^0_t))
	\right) \in [ 2 \epsilon, 4 \epsilon],
	\quad
	d_{\mathbb{F}}^J
	\left( (\mathcal{Z}_I,(\mu_t)),
	(\Y^1_I,(\mu^1_t))
	\right)\ge 2 \epsilon,
\end{equation*}
which is a contradiction to the definition of $ \epsilon$.

\end{proof}

We have the following heat kernel concentration bound. This result also holds for general $4$-dimensional singularity models under the additional assumption of bounded curvature.

\begin{Lemma} \label{lem: H_n-center at a close previous time}
Let $(M^4,g(t),f(t))$, $t\in\mathbb{R}$, be a $4$-dimensional steady gradient Ricci soliton singularity model that satisfies the global non-collapsedness condition $\mathcal{N}_{x,t}(\tau) \geq -Y$ for all $(x,t) \in M \times \mathbb{R}$, $\tau > 0$, where $Y<\infty$ is some uniform constant. 
Suppose we normalize the metric so that 
$R+|\nabla f|^2=1.$
Let $x_0\in M$ and denote $\mu_{t}: = \nu_{x_0,0;t}$ for each $t<0.$ 
Suppose that 
$-A < s<t < -1$, 
$t-s< \delta$, and $(z,t)$ is an $H_4$-center of $(x_0,0).$
If $\delta < 
\bar\delta(Y,A),$ then
\[
    \mu_{s} \big( B(z,t,8\sqrt{H_4|t|} \, ) \big)
    \ge 1/2.
\]
\end{Lemma}
\begin{proof}

\textbf{Claim:} For any $y_1,y_2\in M,$
\[
    d_{t}(y_1,y_2)
    - d^{\vts g_{s}}_{W_1}
    ( \nu_{y_1,t;s},
    \nu_{y_2,t;s})
    <
    \Psi( \delta|Y),
\]
where 
$\Psi(\delta|Y)$ 
depends on $\delta,Y$ and $\Psi(\delta|Y)\to0$ as $\delta\to 0$ for each fixed $Y.$
\medskip \\
\textit{Proof of the claim.}
Since $R \leq 1$ on $M \times \mathbb{R}$, we can use Perelman's Harnack inequality \cite[9.5]{Per03} to deduce that the conjugate heat kernel $K(y,t;\cdot, \cdot)$ based at any $(y, t) \in M \times \{ t\}$ satisfies
\begin{equation}  \label{eq_HKlower}
K(y,t;y,s) 
\geq (4\pi (t-s))^{-n/2} \exp \left(-\frac1{2\sqrt{t-s}} \int_{s}^t \sqrt{t-t'} \, R(y,t')dt' \right)
\geq (4\pi (t-s))^{-n/2} 
e^{-(t-s)/3}
.
\end{equation}
On the other hand, \cite[Theorem 7.2]{Bam20a} implies that for any $H_4$-center $(z',s)$ of $(y,t)$ we have
\begin{equation} \label{eq_HKupper}
K(y,t;y,s) \leq C(Y) (t-s)^{-n/2} \exp \left( - \frac{d^2_s(y,z')}{9(t-s)} \right). 
\end{equation}
Combining \eqref{eq_HKlower} and \eqref{eq_HKupper} implies 
\begin{equation*}
    \frac{d_s^2(y,z')}{9(t-s)}
    \le \ln  C(Y) + (t-s)
    \le \ln C(Y)+\delta,
\end{equation*}
which yields
a distance bound of the form
\[ d_s(y,z') 
\leq C(Y) \sqrt{\delta}. 
\]
So
\begin{equation} \label{eq_dW1ynuy}  
d_{{\rm W}_1}^{\vts g_{s}}
    (\delta_{y}, \nu_{y,t;s})
    \leq 
    d_{{\rm W}_1}^{\vts g_{s}}
    (\delta_{y}, \delta_{z'})
    + d_{{\rm W}_1}^{\vts g_{s}}
    (\delta_{z'}, \nu_{y,t;s})
    \leq d_s(y,z') + \sqrt{H_n (t-s)}
    \leq 
    C(Y)\sqrt{\delta},
\end{equation}
where the latter denotes some generic constant.

Applying \eqref{eq_dW1ynuy} for two points
$y_1,y_2\in M$ yields
\begin{align*}
d_s(y_1,y_2) 
&= d^{g_s}_{W_1}(\delta_{y_1},  \delta_{y_2}) \\ &
\leq d^{g_s}_{W_1}(\delta_{y_1}, \nu_{y_1,t;s})
+d^{g_s}_{W_1}(\nu_{y_1,t;s}, \nu_{y_2,t;s})
+ d^{g_s}_{W_1}(\nu_{y_2,t;s}, \delta_{y_2}) \\
&\leq 
2C(Y) \sqrt{\delta}
+ d^{g_s}_{W_1}(\nu_{y_1,t;s}, \nu_{y_2,t;s}) 
.
\end{align*}
Hence
\[
    d_{{\rm W}_1}^{\vts g_s}(\nu_{y_1,t;s},\nu_{y_2,t;s})
     \ge d_s(y_1,y_2) - 
    C(Y)\sqrt{\delta}.
\]
Let $\Phi_t$ be the 1-parameter
family of diffeomorphisms generated by $-\nabla f.$ 
Then
\[
d_s(y_1,y_2)=
d(\Phi_{s}(y_1),\Phi_{s}(y_2)).
\]
Since 
\begin{equation} \label{ineq: dist}
    d(\Phi_s(x),\Phi_t(x))
    \le \int_s^t |\nabla f|(\Phi_r(x))\, dr
    \le (t-s) 
\end{equation}
by $|\nabla f| \leq 1$, we have 
\begin{align*}
    d_t(y_1,y_2)
    -d_s(y_1,y_2)
& \le d(\Phi_t(y_1),\Phi_s(y_1))
+ d(\Phi_t(y_2),\Phi_s(y_2))
\le 2(t-s) < 2 \delta.
\end{align*}
Thus
\[
d_{t}(y_1,y_2)
    - d^{\vts g_{s}}_{W_1}
    ( \nu_{y_1,t;s},
    \nu_{y_2,t;s})
\le d_t(y_1,y_2)
- d_s(y_1,y_2) +
C(Y)\sqrt{\delta}
\le 
\Psi(\delta|Y).
\]
We have finished the proof of the claim.\smallskip

We can now apply \cite[Lemma 4.18]{Bam20b} with $W$ therein equal to $M$ since 
${\rm Var}(\mu_{t'}) \le 
H_4 A$ for $t'\in [-A,0]$. 
Thus there is a metric space $Z$ 
with embeddings $\varphi_s:(M,d_s)\to Z$ and
$\varphi_t:(M,d_t)\to Z$ such that
\[
    d_Z(\varphi_s(z),\varphi_t(z))
    \le 
    \Psi(\delta|Y,A),
\]
and
\[
    d^Z_{{\rm W}_1}((\varphi_s)_*\mu_s,
    (\varphi_t)_*\mu_t)
    \le \Psi( \delta|Y,A).
\]

Since $(M,g(t),f(t))$ is a steady soliton, 
by \eqref{ineq: dist},
\[
    B(z,s,7\sqrt{H_4|t|})
    \subset B(z,t,8\sqrt{H_4|t|}) 
\]
if $\delta<\bar \delta.$ 
Let $\eta$ be the cutoff function on $Z$ defined by
\[
    \eta(x) = \big( 1 - d_Z(x,B_Z(\varphi_{t}(z),5\sqrt{H_4|t|})) \big)_{+},
\]
which is $1$-Lipschitz and has compact support. 
Then
\begin{align*}
    \mu_{s}  \big( B & (z,t,8\sqrt{H_4|t|}) \big)
    \ge \mu_{s} \big( B(z,s,7\sqrt{H_4|t|}) \big) 
    = (\varphi_{s*}\mu_s)
    \big(B_Z(\varphi_s(z),7\sqrt{H_4|t|})\big)\\
    \ge&   \,
    (\varphi_{s*}\mu_{s}) \big( B_Z(\varphi_{t}(z), 6\sqrt{H_4|t|}) \big)
    \ge \int_Z \eta \, d(\varphi_{s*}\mu_{s})\\
    \ge & \, \int_Z \eta \, d(\varphi_{t*}\mu_{t}) -
    \Psi(\delta|Y,A)
    \ge \mu_{t} \big( B(z, t, 5\sqrt{H_4|t|}) \big) -
    \Psi(\delta|Y,A)
    \ge 1/2,
\end{align*}
since 
$\Psi(\delta|Y,A)\to 0$ as $\delta\to0$ for fixed $Y$ and $A$.
\end{proof}

When a tangent flow at infinity is $(\mathbb{S}^3/\Gamma) \times \mathbb{R}$, we obtain a canonical neighborhood-type result.
The idea of the proof is that in lieu of proving continuity of $H_n$-centers (which are not unique) in the variable $\lambda$, we show an overlapping property for $\epsilon$-necks centered at $H_n$-centers.

\begin{Proposition}\label{prop: steady model eps neck}
Suppose that a $4$-dimensional steady gradient Ricci soliton singularity model $(M^4,g(t),f(t))$ has a tangent flow at infinity isometric to $(\mathbb{S}^3/\Gamma) \times \mathbb{R}$.
Then, for any $\epsilon>0$, there exists a compact set 
$K_{\epsilon} \subset M$ such that any 
$x \in M-K_{\epsilon}$ is the center of an $\epsilon$-neck with respect to $g=g(0)$.

\end{Proposition}

\begin{proof}
By Proposition \ref{lem: uniqueness of tangent flow}, 
there exists a finite subgroup $\Gamma$ of $O(4)$ such 
that each tangent flow at infinity of $(M,g(t))$ is $((\mathbb{S}^{3}/\Gamma) \times \mathbb{R},g_{\rm cyl})$,
where
\begin{equation*}
	g_{\rm cyl} = 4 g_{\mathbb{S}^3/\Gamma} + g_{\mathbb{R}}.
\end{equation*}
Let $\lambda>0$, let $(z_{\lambda},-1/\lambda)$ be an $H_4$-center of $(x_0,0)$, and define $g_{\lambda}(t) = \lambda g(t/\lambda)$.
By the above,
there exist $ \epsilon= \epsilon(\lambda)>0$ and 
a diffeomorphism $\Psi_{\lambda}: B^{\rm cyl}_{1/ \epsilon}\to 
B(z_{\lambda}, 1/ \epsilon; g_{\lambda}(-1))$ such that $ \lim_{\lambda\to 0}\epsilon(\lambda)=0$ and
\begin{equation*}
	\| \Psi_{\lambda}^* g_{\lambda}(-1) - g_{\rm cyl} \|_{C^{[1/ \epsilon]}
	( B^{\rm cyl}_{1/ \epsilon})}
	\le \epsilon ,
\end{equation*}
where $B^{\rm cyl}_{1/ \epsilon}$ denotes a ball of radius $1/ \epsilon$ in $((\mathbb{S}^{3}/\Gamma) \times \mathbb{R},g_{\rm cyl})$.
That is, $z_{\lambda}$ is the center of an $\epsilon$-neck in $(M,g_{\lambda}(-1))$. 
Note that $g_{\lambda}(-1) = \lambda \Phi_{-1/\lambda}^*g$, where $g:=g(0)$ and 
$\Phi_t:M\to M$ 
is the $1$-parameter group of diffeomorphisms generated by  $-\nabla_{g} f.$ We have the composition of diffeomorphisms
\begin{equation*}
	B^{\rm cyl}_{1/ \epsilon} \xrightarrow{\Psi_{\lambda}}
     B(z_\lambda, 1/ \epsilon \,;g_{\lambda}(-1))
	\xrightarrow{\Phi_{-1/\lambda}}
	B\big( w_{\lambda}, 1/(\sqrt{\lambda} \epsilon)  ; g \big)
	=: \mathfrak{N}_{\lambda},
\end{equation*}
where $ w_{\lambda}: =\Phi_{-1/\lambda}(z_{\lambda})$. 
So
\begin{equation*}
	\| \lambda (\Phi_{-1/\lambda}\circ \Psi_{\lambda})^* g - g_{\rm cyl} \|_{
	C^{[1/ \epsilon]}(B^{\rm cyl}_{1/ \epsilon}) }
	\le \epsilon.
\end{equation*}
In particular,
\begin{equation*}
	|{\Rm}_g|(x) \sim c \lambda\quad
	\text{for all}\, x\in \mathfrak{N}_{\lambda}.
\end{equation*}

Choose $\bar \lambda>0$ to be small enough so that 
if $ \lambda \le \bar \lambda$, 
then
$ \epsilon(\lambda) < 10^{-6}$ and
\begin{equation*}
	V_{\lambda} := B \big(z_{\lambda}, 10\sqrt{H_4} \,; g_{\lambda}(-1) \big)
	= B \big(z_{\lambda}, 10\sqrt{H_4/\lambda}\,; g(-1/\lambda) \big)
\end{equation*}
is diffeomorphic to the corresponding ball in $(\mathbb{S}^3/\Gamma) \times \mathbb{R}$.
Write
\begin{equation}\label{eq: ball around center}
	U_{\lambda} : = B \big(w_{\lambda}, 10\sqrt{H_4/\lambda}\, ;g \big)
	= \Phi_{-1/\lambda} (V_{\lambda}).
\end{equation}
We will next show that
\begin{equation*}
	M - K_0 \subset \bigcup_{\lambda>0} 
	10 U_{\lambda}
\end{equation*}
for some compact set $K_0$, 
where 
we denote by 
\[
\alpha B(x,r;g) := B(x,\alpha r;g)
\]
for any $\alpha>0.$ This suffices to show that every point outside of $K_0$ is the center of an $\epsilon$-neck.\medskip

\textbf{Claim:} For any $\lambda_0>0$, there is a $ \delta(\lambda_0)>0$ such that
if $|\lambda - \lambda_0|< \delta$, then
\begin{equation*}
U_{\lambda}\cap U_{\lambda_0} \neq \emptyset.
\end{equation*}
\begin{proof}{Proof of the claim.}
Set
\[
    V'_{\lambda}: = \frac{4}{5} V_{\lambda} : = B\big(z_{\lambda},8\sqrt{H_4}\, ; g_{\lambda}(-1)\big),\quad
    U'_{\lambda} : = \frac{4}{5} U_{\lambda} : = \Phi_{-1/\lambda}(V'_{\lambda}) = B\big(w_{\lambda}, 8\sqrt{H_4/\lambda}\,;g\big).
\]
Suppose, for a contradiction, that there exist $ \lambda_0>0$ and a sequence $\lambda_j\to \lambda_0$ such that
\begin{equation*}
	U_{\lambda_0}\cap U_{\lambda_j} = \emptyset .
\end{equation*}
By applying the diffeomorphism $\Phi_{1/\lambda_0}$ to this, we obtain 
\begin{equation}\label{eq: V's don't intersect}
	V_{\lambda_0} \cap \Phi_{ \frac{1}{\lambda_0} - \frac{1}{\lambda_j}}(V_{\lambda_j})
	=\emptyset.
\end{equation}
For any sufficiently small $\beta>0,$ there exists $\bar j=\bar j(\beta,\lambda_0)$ such that for $j\ge \bar j,$
\[
    \delta_j := \frac{1}{\lambda_0} - \frac{1}{\lambda_j} \in (-\beta,\beta).
\]
For each $x\in V'_{\lambda_j},$ by definition,
\[
    d(x,z_{\lambda_j};g_{\lambda_j}(-1)) < 8 \sqrt{H_4}.
\]
Then
\[
    d(\Phi_{-\delta_j}(x), x; g_{\lambda_j}(-1))
    \le \left|
        \int_{-\delta_j}^0 |\nabla f|_{g_{\lambda_j}(-1)}(\Phi_s(x))
        \, ds
    \right|
    \le |\delta_j| \sqrt{1/\lambda_j} \le \beta \sqrt{\beta + 1/\lambda_0} < 1
\]
if $\beta<\bar \beta(\lambda_0)$. Thus \eqref{eq: V's don't intersect} yields
\[
    V'_{\lambda_j}
    \subset \Phi_{\delta_j}(V_{\lambda_j}),\quad
    \text{and hence}\quad
    V'_{\lambda_0} \cap V'_{\lambda_j} = \emptyset.
\]

Now, 
the key to the proof is that 
by 
the Gaussian concentration estimate of
\cite[Proposition 3.13]{Bam20a},
\begin{equation*}
	\nu_{x_0,0;-1/\lambda_0}(V'_{\lambda_0}) \ge 1 - \frac{1}{64} >0.9.
\end{equation*}
We may assume $-1/\lambda_0 < - 1/\lambda_j$ as the other case can be proved similarly.
By Lemma \ref{lem: H_n-center at a close previous time},
\begin{equation*}
	\nu_{x_0,0;-1/\lambda_0}
	(V'_{\lambda_j}) 
	\ge 1/2,
\end{equation*}
for sufficiently large $j$,
which is a contradiction to the fact that $V'_{\lambda_0} \cap V'_{\lambda_j} = \emptyset$. 
This proves the claim.
\end{proof}

By Munteanu and Wang \cite{MW11}, 
$M$ is connected at infinity if it does not
split for \emph{smooth} steady solitons. 
We include in the appendix a proof 
of their result %
for the case of smooth $4$-orbifolds with isolated singularities assuming that the tangent flow at infinity is $3$-cylindrical.
Thus
$ M-U_{\lambda}$ has
two components when $\lambda<\bar \lambda$.
Let $W^\infty_\lambda$ be the 
unbounded component of $M-U_{\lambda}$
and let $W^0_{\lambda} = M-W^{\infty}_\lambda,$
which is clearly bounded.

Now let $K_0 = \overline{W_{\bar \lambda}^0}.$
Then $K_0$ is compact.
Fix $x\notin K_0$. Consider
\begin{equation*}
	\Lambda := \big\{
		\lambda\in (0,1): 
		x\in W^\infty_{\lambda} 
	\big\}.
\end{equation*}
Let $\lambda_0 = \inf \Lambda$. 
We claim that
$\lambda_0\in (0,\bar \lambda]$.
In fact, $\lambda_0\le \bar \lambda$ directly
follows from the definition.
If $\lambda_0 = 0,$ then
there is a sequence $\lambda_j\to 0$ such that
$x\in W^\infty_{\lambda_j}$ and thus
there is a sequence $y_j\in \partial W^\infty_{\lambda_j}\subset \partial U_{\lambda_j}$ that stays bounded.
By passing to a subsequence, we may assume that
$y_j\to y$ for some point $y\in M$. Then $|{\Rm}|(y)=
\lim_{j\to \infty}|{\Rm}|(y_j)
\le \lim_{j\to \infty} C_n \lambda_j= 0,$
which is a contradiction to the assumption that 
$R>0$ on $M.$

By definition, there exists $\lambda_1\ge \lambda_0$ such that $ \lambda_1\in \Lambda $ and 
$\lambda_1-\lambda_0 < \delta(\lambda_0)/2$. 
Pick $\lambda_2\in (0,\lambda_0)$ such that $\lambda_0 - \lambda_2 < \delta/2$.
We proved above that
\begin{equation*}
	U_{\lambda_1}\cap U_{\lambda_2} \neq \emptyset.
\end{equation*}
Since $x\in W_{\lambda_2}^0,$ 
we have
$x\in 10U_{\lambda_1}$.
Thus
\begin{equation*}
	M - K_0 \subset \bigcup_{\lambda>0} 
	10 U_{\lambda}.
\end{equation*}
As $10 \ll \frac{1}{10\epsilon(\lambda)}$ and
$10 U_{\lambda}$ lies in the middle of the neck region $\mathfrak{N}_{\lambda} := \frac{1}{10\epsilon(\lambda)} U_{\lambda}$, we have that every point outside of $K_0$ is the center of an $\epsilon$-neck.
This completes the proof of the proposition.
\end{proof}

As a result, we can see that if $(M^4,g(t),f(t))$ is a steady gradient Ricci soliton singularity model whose tangent flow at infinity is 
$(\mathbb{S}^3/\Gamma)\times \mathbb{R}$, then it is \textbf{asymptotically (quotient) cylindrical} in the following sense: for any sequence $x_j\to \infty$,
\[
    (M, R(x_j)g, x_j) \to ((\mathbb{S}^3/\Gamma)\times \mathbb{R},\bar g, x_\infty)
\]
(without passing to a subsequence),
where $\bar g$ is the rescaling of the standard cylindrical metric with scalar curvature
$R(\bar g) = 1.$ In fact, for any $x_j\to \infty, $ by the last proposition, 
$x_j\in 10 U_{\lambda_j}$ for some $\lambda_j>0.$ Since 
$R(x_j) = 1.5 \lambda_j + o(1)$ and $10 U_{\lambda_j}\subset \mathfrak{N}_{\lambda_j}$ is an $\epsilon$-neck, we 
have the convergence.

By a result of Munteanu  and Sesum \cite[Corollary 5.2]{MS13},
whose proof applies in the orbifold setting (see also Wu \cite[Theorem 1.1]{Wu13}), 
we have the following.

\begin{Proposition}\label{prop: Theorem 7.8}
If $(M^{n},g,f)$ is a complete noncompact non-Ricci-flat steady gradient Ricci soliton
and $o\in M$, then there exists a constant $C$ such that
for $r\geq 1$,
\begin{equation}
r -C\sqrt{r} 
\le \sup_{\partial B_r(o)} f
\le r + C.
\label{upper bound for inf f dist sph steady}
\end{equation}
\end{Proposition}

We prove an a priori curvature estimate.
\begin{Lemma}\label{lem: cylindrical ASCR=oo}
If a complete steady gradient Ricci soliton $(M^n,g,f)$ is asymptotically cylindrical,
then 
\[
    \lim_{x\to \infty} R(x) r^2(x)
    =\infty,
\]
where $r(x)= d(x,o)$ and $o$ is a fixed point.
\end{Lemma}
\begin{proof}
Suppose that there is a sequence $x_j\to \infty$ such that $R(x_j)r^2(x_j)\le A^2$ for some constant $A<\infty.$
Since $(M,g)$ is asymptotically cylindrical, 
there is a sequence $A_j\to \infty$ such that
$\Big(B(x_j, \frac{A_j}{\sqrt{R(x_j)}};g),R(x_j)g, x_j\Big)$ converges
to $((\mathbb{S}^{n-1}/\Gamma)\times \IR,\bar g, x_\infty)$ in the pointed Cheeger--Gromov sense.
Since $d_{R(x_j)g} (o,x_j) \leq A$ and since the scalar curvature is constant on the cylinder, we have that $\frac{R(o)}{R(x_j)} = 1 + \operatorname{o}(1)$.
Now, letting $j\to \infty$, we obtain that $R(o)=0,$ which is a contradiction
to the fact that $g$ is not Ricci flat (since it is asymptotically cylindrical).

\end{proof}


\begin{Lemma}\label{lem: f over r goes to 1}
If 
a steady gradient Ricci soliton
$(M^{n},g,f)$ is asymptotically cylindrical, then $\lim
_{x\rightarrow\infty}\frac{f(x)}{r(x)}=1$.
\end{Lemma}

A proof of this is in
\cite[Theorem 2.1]{CDM20} because 
$\Ric(\nabla f,\nabla f)\ge 0$ 
outside a compact set by \eqref{Df DR little r}.
For completeness, we include an alternative argument.

\begin{proof}
Fix $o\in K_{\epsilon}$, where $K_{\epsilon}$ is given by Proposition \ref{prop: steady model eps neck} and $\epsilon >0$ is sufficiently small. 
Let $\rho_{0}$ be sufficiently large so that $x \in M - K_{\epsilon}$ whenever $r(x) \geq \rho_0$.
By Proposition \ref{prop: Theorem 7.8}, there exists 
$y_{0}\in\partial B_{\rho_{0}}(o)$
such that 
$\rho_0 - C\sqrt{\rho_0} \le f(y_0) \le \rho_0 + C$. Moreover,
we have that $y_{0}$ is the
center of an $\epsilon$-neck 
$\mathfrak{N}$. 
Let $\phi:B_{1/\epsilon}^{\rm cyl}\to \mathfrak{N}$ be a diffeomorphism such that
$R(y_0)\phi^*g$ is $\epsilon$-close to $\bar g,$
where $B_{1/\epsilon}^{\rm cyl}$ is a ball of radius $1/\epsilon$ in 
$(\mathbb{S}^3/\Gamma)\times \IR.$ 
We denote by 
$S_0 = \phi\left( (\mathbb{S}^{n-1}/\Gamma)\times\left\{ 0\right\} \right)$ the center sphere of $\mathfrak{N}$.
We know $S_0$ is diffeomorphic to $\mathbb{S}^{n-1}/\Gamma$ and 
\begin{equation}\label{eq: diam est S0}
    {\rm diam} \,S_0 \le \frac{C_n}{\sqrt{R(y_0)}}
    \le \operatorname{o}(\rho_0)\quad
    \text{as } \rho_0\to \infty,
\end{equation}
by Lemma \ref{lem: cylindrical ASCR=oo}.

Let $L=10 C_n/\sqrt{R(y_0)}$. Set $S_1 = \phi\left( (\mathbb{S}^{n-1}/\Gamma)\times\left\{ -L\right\} \right),
S_2 =\phi\left( (\mathbb{S}^{n-1}/\Gamma)\times\left\{ L\right\} \right)$.\smallskip

\noindent \textbf{Claim.} 
\[\partial B_{\rho_0}(o)\subset \phi\left( (\mathbb{S}^{n-1}/\Gamma)\times [-L,L] \right)=:
\mathfrak{N} (L).
\]

We know that $S_1,S_2$ are both diffeomorphic to $\mathbb{S}^{n-1}/\Gamma$ and that
they share the same diameter estimate \eqref{eq: diam est S0} as that of $S_0.$
Suppose that the claim is not true and there is some $y_1\in \partial B_{\rho_0}(o)- \mathfrak{N}(L).$ 
We may assume $y_1$ lies in the bounded component of $M- \mathfrak{N}(L)$ since the proof of the other case is similar.
Since $(M,g)$ is asymptotically cylindrical and has only one end,
when $\rho_0$ is sufficiently large, $\partial B_{\rho_0}(o)$ is connected;
hence 
there exists $y_2\in \partial B_{\rho_0}(o)\cap S_1.$ 
Suppose $y_0=\phi(\bar y_0,0)$ for some
$\bar y_0\in \mathbb{S}^{n-1}/\Gamma.$
Let $z_0 = \phi(\bar y_0,-L).$ Then
$d(z_0,y_0) \le L+\epsilon$ and
\[
    \rho_0 = r(y_0)
    \ge r(z_0) + L - \epsilon
    \ge r(y_2) - {\rm diam} S_1 + L - \epsilon
    = \rho_0 - {\rm diam} S_1 + L - \epsilon,
\]
which is a contradiction. This proves the claim.\smallskip

It follows from the claim that for each $y\in \partial B_{\rho_0}(o),$
\[
    d(y,y_0) \le L + \diam S_0 \le C_n/\sqrt{R(y_0)}
    \le o(\rho_0).
\]
As $|\nabla f|\le 1,$ for each $y\in \partial B_{\rho_0}(o),$
\[
    f(y) \ge f(y_0) - d(y,y_0)
    \ge \rho_0 - C\sqrt{\rho_0} - o(\rho_0).
\]
Hence $\lim_{x\to \infty} f(x)/r(x) = 1.$

\end{proof}

We have the following result of Brendle \cite{Bre14}. For completeness, we include his proof.

\begin{Proposition}\label{lem: asymp quot cyl R asymptotics}
If $(M^{n},g,f)$ is asymptotically 
cylindrical, then 
\begin{equation}\label{eq: fR about n minus 1 over 2}
fR=\frac{n-1}{2}+\operatorname{o}(1).
\end{equation}
This implies that 
\begin{equation*}
   d(x,p) |{\Rm}_g|(x) \to c_n\quad \text{as }\, x \to \infty.
\end{equation*}
\end{Proposition}

\begin{proof}
By Lemma \ref{lem: f over r goes to 1},
there exists a constant $C$ such that
\[
C^{-1}r(x)- C
  \leq f(x)\leq r(x) +C.
\]
Since $R=\operatorname{o}\left(  1\right)  $, we have%
\begin{equation}
\left\vert \nabla f\right\vert ^{2}=1+\operatorname{o}\left(  1\right)  .
\label{Df sq is 1 plus o 1}%
\end{equation}
Because an exact $n$-dimensional quotient cylinder $((\mathbb{S}^{n-1}%
/\Gamma)\times\mathbb{R},\bar{g})$ satisfies the scale-invariant identities
$R_{\bar{g}}^{-2}\Delta_{\bar{g}}R_{\bar{g}}=0$ and $(n-1)R_{\bar{g}}%
^{-2}|\operatorname{Ric}_{\bar{g}}|^{2}=1$, we have%
\begin{equation}
\Delta R=\operatorname{o}\left(  R^{2}\right)  \quad\text{and}\quad
(n-1)\left\vert \operatorname{Ric}\right\vert ^{2}=R^{2}+\operatorname{o}%
\left(  R^{2}\right)  . \label{Exercise Prove This Remember}%
\end{equation}
Indeed, if (\ref{Exercise Prove This Remember}) is not true, then there exists
a sequence of points tending to infinity about whose rescalings limit to a
solution which is not a cylinder $(\mathbb{S}^{n-1}/\Gamma)\times\mathbb{R}$.
Hence standard formulas imply that
\begin{equation}
- \left\langle \nabla f,\nabla
R\right\rangle 
=\Delta R+2|{\operatorname{Ric}}|^{2}=\frac{2}{n-1}R^{2}%
+\operatorname{o}\left(  R^{2}\right) . 
\label{Df DR little r}%
\end{equation}
Using this and (\ref{Df sq is 1 plus o 1}), we compute that%
\begin{equation}
-\left\langle \nabla f,\nabla\left(  R^{-1} - \frac{2}{n-1}f\right)
\right\rangle 
=\frac{2}{n-1}\left\vert \nabla f\right\vert ^{2}
+
\frac
{\left\langle \nabla f,\nabla R\right\rangle }{R^{2}}=\operatorname{o}\left(
1\right)  . \label{This and that just 3rd}%
\end{equation}

Now we show that integrating this over integral curves to $-\nabla f$ yields
the proposition. 
Choose $r_{0}$ so that 
\begin{equation}\label{eq: grad f at least half}
|\nabla f|^{2}\geq\frac{1}{2} \quad \text{on }\, M-B_{r_{0}}(o) .
\end{equation}
Let $x\in M-B_{r_{0}}(o)$ and let
$\sigma:(-\infty,\infty)\rightarrow M$ be the integral curve to
$- \nabla f$ with $\sigma(0)=x$. 
By \eqref{eq: grad f at least half}, there exists a smallest $u_{0}>0$ such
that $\sigma(u_{0})\in\bar{B}_{r_{0}}(o)$. Define $\phi=R^{-1}- \frac{2}{n-1}%
f$. We have%
\begin{align}
\phi(x)-\phi(\sigma(u_{0}))  &  = \int_{0}^{u_{0}}\left\langle \nabla
\phi,\sigma^{\prime}(u)\right\rangle
du\label{This justifies the third equality in}\\
&  = - \int_{0}^{u_{0}}\left\langle \nabla f,\nabla\left(  R^{-1} - \frac{2}%
{n-1}f\right)  \right\rangle (\sigma(u))du\nonumber\\
&  =\operatorname{o}\left(  u_{0}\right) \nonumber\\
&  =\operatorname{o}\left(  r(x)\right)  .\nonumber
\end{align}
Note that, for $u\in [0,u_0],$
\[
f(\sigma(u))-f(\sigma(u_{0}))= - \int_{u_{0}}^{u}\left\vert \nabla
f\right\vert ^{2}(\sigma(t))dt\geq  \frac{1}{2}\left(  u_{0}-u\right)  ,
\]
so that $d(\sigma(u),o)\geq c\left(  u_{0}-u\right)  -C$, where $c$ and $C$
are independent of $x$ and $u$. This and (\ref{This and that just 3rd})
justify the third equality in (\ref{This justifies the third equality in}) and
thus complete the proof of the proposition.

\end{proof}

Recall that $w_\lambda = \Phi_{-1/\lambda}(z_\lambda)$, where $(z_{\lambda},-1/\lambda)$ is an $H_n$-center of $(x_0,0).$
We have that
$R(w_{\lambda})=1.5\lambda + o(1)$ as $\lambda\to 0.$ 
By Proposition \ref{lem: asymp quot cyl R asymptotics}, 
\[
    \lim_{\lambda\to 0} \lambda f(w_{\lambda}) =1.
\]

We have the following result, which was proved by Xiaohua Zhu and the third author in dimension $4$ \cite[Theorem 1.5]{DZ20}.
\begin{Proposition}
If $(M^n,g,f)$ is a complete steady gradient Ricci soliton that is asymptotically cylindrical, 
then
there exists a compact set $K$ such that $(M-K,g)$ has positive curvature operator and satisfies
\begin{equation}\label{eq: Rm linear decay}
    C^{-1} d(x,p)^{-1} \leq |{\Rm}_g|(x) \leq C d(x,p)^{-1} \quad \text{for }\, x \in M-K.
\end{equation}
\end{Proposition}

\begin{proof}

For simplicity, assume that $\dim M=n=4$. The proof of the general case is the same.

Since $(M^4,g)$ is asymptotically cylindrical, for any sequence $x_j\to \infty$, 
\[
\big( M,R(x_j)g(t/R(x_j)),(x_j,-1) \big)\to \big( (\mathbb{S}^3/\Gamma)\times \IR, \bar g(t), (x_\infty,-1) \big)
\]
in the pointed Cheeger--Gromov sense. Under this convergence of metrics, the rescaled vector fields $R^{-1/2}(x_j)\nabla_g f$ converge in $C_{\rm loc}^\infty$ to the vector field $\partial_s$ on $(\mathbb{S}^3/\Gamma)\times \mathbb{R}$, where $s$ is the coordinate on the $\IR$-factor; this fact can be proved in the same way as in Brendle \cite[Proposition 2.5]{Bre13}.

The main issue is to show that sectional curvatures of planes containing the radial directions of the $\epsilon$-necks are positive.
To this end, define
\[
A_{jk} := \sum_{i,\ell=1}^4 R_{ijk\ell }\nabla _{i}f\nabla _{\ell }f ,
\]
where our curvature sign convention is such that for orthonormal vectors $v,w$, $R_{ijk\ell }v_iw_jw_kv_\ell$ is the sectional curvature of the plane spanned by $v,w$.
By standard equations for steady solitons, we have
\begin{align*}
A_{jk}
& =\nabla _{i}f\left( \nabla _{i}R_{jk}-\nabla _{j}R_{ik}\right)  \\
& =\Delta R_{jk}+2R_{ijk\ell }R_{i\ell }-\nabla _{i}f\nabla _{j}R_{ik} \\
& =\Delta R_{jk}+2R_{ijk\ell }R_{i\ell }-\nabla _{j}\left( \nabla
_{i}fR_{ik}\right) +\nabla _{j}\nabla _{i}fR_{ik} \\
& =\Delta R_{jk}+2R_{ijk\ell }R_{i\ell }-\frac{1}{2}\nabla _{j}\nabla
_{k}R-R_{ji}R_{ik}.
\end{align*}
Indeed, this is the steady version of a formula Hamilton derived for expanding solitons in \cite[\S 3]{Ham93a}.
Since our steady soliton is asymptotically cylindrical, we have
$|\nabla {\Ric}|=\operatorname{o}(R^{3/2}), |\Delta {\Ric}|=\operatorname{o}(R^{2})$ and $|\nabla^2R| =\operatorname{o}(R^{2})$.

Moreover, for the round cylinder $\bar{g}$ 
with 
scalar curvature
$R(\bar g)=1$ and
local
coordinates $\{ \bar x^{1}, \bar x^{2}, \bar x^{3}\}$ on $\mathbb{S}^{3}$ and $\bar x^{4}$
the Euclidean coordinate for $\mathbb{R}$, we have for $1\leq j,k \leq 3$ that
\begin{align*}
   \bar{A}_{jk} &  :=   \,  \bar \Delta \bar R_{jk}+2\bar R_{ijk\ell }\bar R_{i\ell }-\frac{1}{2}\bar \nabla _{j} \bar\nabla
_{k}\bar R-\bar R_{ji}\bar R_{ik}
= \frac{1}{9}\bar g_{jk}
\end{align*}
since $\bar{R}_{ijk\ell}=\frac{1}{6}(\bar g_{i\ell} \bar g_{jk} - \bar g_{ik} \bar g_{j\ell})$ and $\bar{R}_{ijk4}=0$ for $1\leq i,j,k,\ell\leq3$. 
Let $\{x^{i}\}$ be local coordinates on $M$ with $\frac{%
\partial}{\partial x^{4}}=\nabla f$ and $x^{1},x^{2},x^{3}$ tangent to the
level sets of $f$.
Thus, we have for $1 \le j,k \le 3$,
\begin{equation*}
R_{4jk4}=
R_{ijk\ell}\nabla_{i}f\nabla_{\ell}f=A_{jk}=\frac{R^{2}}{9}g_{jk}+\operatorname{o}%
(R^{2}).
\end{equation*}
We have for $1\le i,j,k,\ell\le 3,$
\[
    R_{ijk\ell} = \frac{R}{6}\left( g_{i\ell}g_{jk}-g_{ik}g_{j\ell}\right) + \operatorname{o}(R).
\]
Moreover, for $1\le i,j,k\le 3,$
\[
R_{ijk4} = \nabla_j R_{ik} - \nabla_i R_{jk} = \operatorname{o}(R^{3/2}) .
\]

To show that the curvature operator is positive away from a compact set, we consider an arbitrary nontrivial $2$-form $\phi=\sum_{i,j=1}^{3} a_{ij}\, \partial_i\wedge \partial_j + \sum_{k=1}^3 b_k \, \partial_k \wedge \nabla f$ at a point, where $a_{ji}=-a_{ij}$. 
Write $A=(a_{ij}),b=(b_k).$ 
We compute that
\begin{align*}
    \Rm(\phi,\phi)
    &= \sum_{i,j,k,\ell=1}^3 a_{ij}a_{\ell k} R_{ijk\ell}
    + \sum_{j,k=1}^3 b_{j} b_{k} R_{4jk4}
    - 2\sum_{i,j,k=1}^3 a_{ij}b_k R_{ijk4}\\
    &\ge \frac{R}{3}|A|^2 (1- \operatorname{o}(1))
    + \frac{R^2}{9}|b|^2(1-\operatorname{o}(1))
    - |A||b|\operatorname{o}(R^{3/2})\\
    & >0
\end{align*}
outside of a sufficiently large compact set.
x
Finally, \eqref{eq: Rm linear decay} follows from \eqref{eq: fR about n minus 1 over 2}.
\end{proof}

\def \Rc {{\rm Rc}}

\appendix
\section{}

In this appendix, we prove 
that for any $4$-dimensional steady soliton on an orbifold with isolated singularities, if its tangent flow at infinity is $(\mathbb{S}^3/\Gamma)\times\mathbb{R}$, then it has only one end, and all of the singular points must 
lie
in a compact set. 
This is the slight extension of Munteanu and Wang's result stated in the proof of Proposition \ref{prop: steady model eps neck}.
As a consequence, all of our arguments in the previous section
are applicable 
to solitons on orbifolds with isolated singularities satisfying the conditions assumed in our main theorem.

\begin{Theorem}
Let $(M^4,g,f)$ be a steady soliton on an orbifold with isolated singularities such that the tangent flow at infinity of its canonical form is $(\mathbb{S}^3/\Gamma)\times\mathbb{R}$. Then $(M^4,g,f)$ is connected at infinity.
\end{Theorem}

\begin{proof}
Let us fix a point $x_0$ on 
$M$
and let $U_\lambda$, where $\lambda>0$, be the open ball defined 
by
\eqref{eq: ball around center} in the proof of Proposition \ref{prop: steady model eps neck}. By the claim in the proof of Proposition \ref{prop: steady model eps neck}, we have that the open set 
\begin{equation}
    U:=\bigcup_{\lambda<\bar\lambda}10U_\lambda
\end{equation}
is connected and covered by $\epsilon$-necks, and is therefore an end of $(M,g)$, where $\bar \lambda$ is a small positive number defined in the same way as in Proposition \ref{prop: steady model eps neck}. 
For a contradiction, let us assume
that 
$U$ is not the unique end. Then we can find a sequence $\{x_i\}_{i=1}^\infty$ of points in $M$ such that 
\begin{equation}\label{contradictorynonsense}
    d_g(U,x_i)\nearrow\infty.
\end{equation}

Next, we consider the canonical form $(M,g(t))_{t\in(-\infty,0]}$ of the steady soliton in question. By the assumption of the theorem, we have that, fixing any 
$i$,
for any sequence $\lambda_k\searrow 0$ it holds that
$(M,\lambda_kg(\lambda_k^{-1}t),\nu_{x_i,0;\lambda_k^{-1}t})$ converges in the $\mathbb{F}$-sense to $(\mathbb{S}^3/\Gamma)\times\mathbb{R}$. Since $(\mathbb{S}^3/\Gamma)\times\mathbb{R}$ is smooth, this convergence is also smooth. As a consequence, we have that for any $\epsilon>0$, whenever $\lambda$ is small enough, it holds that 
$\big( B(z_{i,\lambda},1/\epsilon;\lambda g(-\lambda^{-1})), \lambda g(-\lambda^{-1}) \big) $ 
is $\epsilon$-close to the corresponding subset of $(\mathbb{S}^3/\Gamma)\times\mathbb{R}$ in the smooth sense, where
\begin{equation}\label{definitionnonsense}
    (z_{i,\lambda},-\lambda^{-1})\  \text{is some $H_n$-center of}\ (x_i,0) \ \text{with respect to}\ g(t).
\end{equation}

Now we fix a small positive number $\epsilon\ll 10^{-6}$ and define $\lambda_{i}$ as follows: 
\begin{enumerate}
    \item For all $\lambda\leq \lambda_i$, $\big(B(z_{i,\lambda},1/\epsilon;\lambda g(-\lambda^{-1})), \lambda g(-\lambda^{-1})\big)$ is $\epsilon$-close to the corresponding subset of $(\mathbb{S}^3/\Gamma)\times\mathbb{R}$ in the smooth sense. Here, as before, $(z_{i,\lambda},-\lambda^{-1})$ is an $H_n$-center of $(x_i,0)$ with respect to $g(t)$.
    \item $\big(B(z_{i,\lambda_i},1/\epsilon;\lambda_ig(-\lambda_i^{-1})), \lambda_ig(-\lambda_i^{-1})\big)$ is not $\epsilon/2$-close to any subset of $(\mathbb{S}^3/\Gamma)\times\mathbb{R}$.
\end{enumerate}
Note that such 
a
positive $\lambda_i$ must exist, since the blow-up limit at any point must be 
Euclidean space or a Euclidean cone. Now we split our argument into two cases.
\\

\noindent\textbf{Case I.} 
\emph{The sequence
$\lambda_i$ is bounded from below, namely, there is a positive number $c$ such that $\lambda_i\geq c$ for all $i$.}

Let $w_{i,\lambda}=\Phi_{-1/\lambda}(z_{i,\lambda})$ for all $i\ge 0$. Then we have 
\begin{equation}\label{nonsense}
    d_g(w_{i,\lambda},w_{0,\lambda})
= d\big( z_{i,\lambda}, z_{0,\lambda};g(-\lambda^{-1}) \big) 
\le d_g(x_i,x_0)+2\sqrt{H_n/\lambda}.
\end{equation}
So
we must have that $w_{i,\lambda}\in U$ when $\lambda$ is small enough. 
Indeed,
suppose this is not true. Recall that $w_{0,\lambda}$ is the center of an $\epsilon$-neck with radius approximately $\sqrt{\lambda^{-1}}$, and this $\epsilon$-neck is contained in $U$. Therefore, any point outside $U$ must be at least
distance
$\epsilon^{-1}\sqrt{\lambda^{-1}}$ away from $w_{0,\lambda}$. This is clearly a contradiction to \eqref{nonsense} when $\lambda$ is small enough.

Arguing in the same way as for the claim in Proposition \ref{prop: steady model eps neck}, for each $i$, we can construct an open set 
$$U_i:=\bigcup_{\lambda<c}10B \big(w_{i,\lambda}, 10\sqrt{H_4/\lambda}\, ;g \big)$$
which is also connected and covered by $\epsilon$-necks. Obviously, $U_i$ is also an end of $(M,g)$ and,  according the the argument in the previous paragraph,  there exists a compact set $K_i\subset M$ such that $U_i\setminus K_i= U\setminus K_i$. 

By the proof of the claim in Lemma \ref{lem: H_n-center at a close previous time}, 
\begin{equation}\label{distancenonsense}
    d_g(w_{i,c},x_i)\le C(Y)/c,
\end{equation}
where $\NN_{x_0,0}(\tau)\ge -Y$ for any $\tau>0.$
By the assumption \eqref{contradictorynonsense} for a contradiction and by (\ref{distancenonsense}), we also have $d_g(U,w_{i,c})\nearrow\infty$. This 
shows that $U_i\not\subset U$. 
On the other hand,
if $U\not\subset U_i$, then their boundaries 
would
intersect. However, by definition, $\partial U_i$ is 
approximately distance
$\sqrt{c^{-1}}$ away from $w_{i,c}$, and this is clearly impossible. In conclusion, we have $U\subset U_i$. By the same argument, we also have $U_i\subset U_j$ if $j\gg i$. Therefore, 
$$U_\infty:=\bigcup_{i\geq 1}U_i$$
is an $\epsilon$-tube with infinite length on both ends, and it must be the whole manifold $M$, for otherwise $M$ is not connected. This also implies that $(M^4,g,f)$ is a steady soliton on a smooth manifold with two ends, which is impossible by Munteanu and Wang \cite{MW11}.
Alternatively, we may also use the closeness to $( \mathbb{S}^3/\Gamma) \times \mathbb{R}$ and deduce that   $\partial_t R(\cdot, t) > c > 0$, which contradicts the fact that all time-slices are isometric and the uniform bound $R \leq 1$.

\smallskip

\noindent\textbf{Case II.} \emph{$\lambda_i$ is not bounded from below.}

Let us consider the sequence of rescaled flows 
\begin{equation}\label{scalingnonsense}
\left\{ \big(M,\lambda_ig(\lambda_i^{-1}t),\nu_{x_i,0;\lambda_i^{-1}t}\big) \right\}_{i=1}^\infty.
\end{equation}
By \cite{Bam20b,Bam20c}, after passing to a subsequence, there is an $\mathbb{F}$-limit $\mathcal X$ whose singular set has space-time Minkowski dimension no greater than $2$. Since $\mathcal X$ is a blow-down limit, the potential functional $f$ of the original steady soliton gives rise to a parallel vector 
field
on the regular part of $\mathcal X$.
Since $f \circ \Phi_t$ is a solution to the heat equation, we may apply \cite[Theorem~15.50]{Bam20c} to $\lambda_i^{1/2} f \circ \Phi_t$ and conclude that $\mathcal X$ splits as $\IR\times \mathcal Y$, where $\mathcal Y$ is a $3$-dimensional metric flow whose singular set has space-time Minkowski dimension no greater than $1$, and hence must be a smooth ancient Ricci flow. Letting $\mathcal X:=\IR\times(N^3_\infty,g_\infty(t))_{t\in(-\infty,0)}$, we then have that $(z_{1,\infty},-1)$ is not the center of an $\frac{\epsilon}{2}$-neck, but $(z_{\lambda,\infty},-\lambda^{-1})$ is the center of an $\epsilon$-neck for all $\lambda<1$, where $(z_{\lambda,\infty},-\lambda^{-1})$ is the limit of the sequence of space-time points $\{(z_{i,\lambda\lambda_i},-\lambda^{-1})\}$ (c.f. \eqref{definitionnonsense}) in the sequence of rescaled flows (\ref{scalingnonsense}). This further shows that $(N_\infty,g_\infty(t))$ is $\epsilon$-close to $\mathbb{S}^3/\Gamma$ for all $t\leq -1$, but not $\frac{\epsilon}{2}$-close to $\mathbb{S}^3/\Gamma$ at $t=-1$. This is clearly a contradiction by Hamilton's theorem \cite{Ham82}.
%
\end{proof}

\bibliography{bibliography}{}

\begin{thebibliography}{CCG{\etalchar{+}}10}

\bibitem[App17]{App17} Appleton, Alexander. \emph{A family of non-collapsed steady Ricci solitons in even dimensions greater or equal to four.}
arXiv:1708.00161v4 (2017).

\bibitem[Bam20a]{Bam20a}
Bamler, Richard~H.  \emph{{Entropy and heat kernel bounds on a Ricci flow background}},
  arXiv:2008.07093v3 (2020).

\bibitem[Bam20b]{Bam20b}
\bysame, \emph{{Compactness theory of the space of super Ricci flows}},
  arXiv:2008.09298v2 (2020).
  
\bibitem[Bam20c]{Bam20c}
\bysame, \emph{{Structure theory of non-collapsed limits of Ricci flows}}, 
  arXiv:2009.03243v2 (2020).
  
\bibitem[BKN89]{BKN89} Bando, Shigetoshi; Kasue, Atsushi; Nakajima, Hiraku. \emph{On a construction of coordinates at infinity on manifolds with fast curvature decay and maximal volume growth.} Invent. Math. \textbf{97} (1989), no. 2, 313--349.

\bibitem[Bre13]{Bre13} Brendle, Simon. \textit{Rotational symmetry of self-similar solutions to the Ricci flow}, Invent. math. \textbf{194} (2013), 731–764.  
 
\bibitem[Bre14]{Bre14} Brendle, Simon. \textit{Rotational symmetry of Ricci solitons in higher dimensions}, J. Diff. Geom. \textbf{97} (2014), no. 2, 191--214.

\bibitem[Bry05]{Bry05} Bryant, Robert. \emph{{Ricci flow solitons in dimension three with
  $SO(3)$-symmetries}},
  http://www.math.duke.edu/$\sim$bryant/3DRotSymRicciSolitons.pdf. (2005).




\bibitem[CCZ08]{CCZ08}Cao, Huai-Dong; Chen, Bing-Long; Zhu, Xi-Ping. \emph{Recent
developments on Hamilton's Ricci flow}, Surveys in differential geometry. Vol.
\textbf{XII}. Geometric flows, 47--112, Int. Press, Somerville, MA, 2008.


\bibitem[Cha19]{Cha19}  Chan, Pak-Yeung. \textit{Curvature estimates for steady gradient Ricci solitons}, Trans. Amer. Math. Soc. \textbf{372} (2019), no. 12, 8985--9008.


\bibitem[CMZ21a]{CMZ21a} Chan, Pak-Yeung; Zilu Ma; Yongjia Zhang. \emph{Ancient Ricci flows with asymptotic solitons.} 
arXiv:2106.06904 (2021).

\bibitem[CMZ21b]{CMZ21b} Chan, Pak-Yeung; Zilu Ma; Yongjia Zhang. \emph{On Ricci flows with closed and smooth tangent flows.} 
arXiv:2109.14763 (2021).


\bibitem[ChN15]{ChN15} Cheeger, Jeff; Naber, Aaron. \emph{Regularity of
Einstein manifolds and the codimension }$4$\emph{ conjecture}, Annals of
Mathematics \textbf{182} (2015), 1093--1165.


\bibitem[CDM20]{CDM20}
Chow, Bennett; Deng, Yuxing; Ma, Zilu. \emph{{On four-dimensional steady gradient Ricci solitons that dimension reduce}}, 
arXiv:2009.11456 (2020).

\bibitem[CFSZ20]{CFSZ20} Chow, Bennett; Freedman, Michael; Shin, Henry; Zhang, Yongjia. \emph{{Curvature growth of some 4-dimensional gradient Ricci soliton singularity models}},
Advances in Mathematics, \textbf{372} (2020), article number 107303.


\bibitem[CLY11]{CLY11} 
Chow, Bennett; Lu, Peng; Yang, Bo.
\emph{Lower bounds for the scalar curvatures of noncompact gradient Ricci solitons}, Comptes Rendus Mathematique Ser. I, \textbf{349} (2011), 1265--1267.


\bibitem[CM21]{CM21} 
Colding, Tobias Holck; Minicozzi, William P.~II.
\emph{Singularities of Ricci flow and diffeomorphisms}, arXiv:2109.06240 (2021).



\bibitem[DZ20]{DZ20} Deng, Yuxing; Zhu, Xiaohua. \emph{Classification of gradient steady Ricci solitons with linear curvature decay},
Science China Mathematics \textbf{63} (2020), 135--154.

\bibitem[Ham82]{Ham82} Hamilton, Richard S. \emph{Three-manifolds with positive Ricci curvature}, J. Differential
Geom. \textbf{17} (1982), no. 2, 255--306.

\bibitem[Ham93a]{Ham93a} Hamilton, Richard S. \emph{The Harnack estimate for the Ricci flow}, J. Diff.
Geom. \textbf{37} (1993), no. 1, 225--243.

\bibitem[Ham93b]{Ham93b} Hamilton, Richard S. \emph{The formation of singularities in the
Ricci flow.\ }Surveys in differential geometry, Vol.\ II (Cambridge, MA,
1993), 7--136, Internat. Press, Cambridge, MA, 1995.


\bibitem[KW15]{KW15} Kotschwar, Brett; Wang, Lu. \emph{Rigidity of asymptotically conical shrinking gradient Ricci solitons}, J. Diff. Geom., \textbf{100} (2015), 55--108.

\bibitem[KW22]{KW22} Kotschwar, Brett; Wang, Lu. \emph{A uniqueness theorem for asymptotically cylindrical shrinking Ricci solitons}, J. Diff. Geom., to appear (2022).

\bibitem[Lai20]{Lai20} Lai, Yi. \emph{A family of 3d steady gradient solitons that are flying wings}, J. Diff. Geom., to appear, arXiv:2010.07272 (2020).

\bibitem[LW21]{LW21} 
Li, Yu; Wang, Bing.
\emph{Rigidity of the round cylinders in Ricci shrinkers}, 
arXiv:2108.03622 (2021).

\bibitem[MZ21]{MZ21} Ma, Zilu; Zhang, Yongjia.
\emph{Perelman's entropy on ancient Ricci flows}, J. Funct. Anal. 281 (2021), no. 9,
Paper No. 109195, 31 pp.; MR4290285.


\bibitem[MS13]{MS13}
Munteanu, Ovidiu; Sesum, Natasa. \emph{On gradient Ricci solitons}, J. Geom. Anal. \textbf{23} (2013),
no.~2, 539--561.

\bibitem[MSW19]{MSW19}
Munteanu, Ovidiu; Sung, Chiung-Jue Anna; Wang, Jiaping.
\emph{{Poisson equation on complete manifolds}}, Adv.
Math. \textbf{348} (2019), 81--145.

\bibitem[MW11]{MW11} 
Munteanu, Ovidiu; Wang, Jiaping.
\emph{{Smooth metric measure spaces with non-negative
curvature}},
Communications in
Analysis and Geometry
\textbf{19} (2011), No. 3, 451--486.

\bibitem[MW15]{MW15} Munteanu, Ovidiu; Wang, Jiaping. \emph{Geometry of shrinking Ricci solitons}. Compos. Math. \textbf{151} (2015), no. 12, 2273--2300. 

\bibitem[MW17]{MW17} Munteanu, Ovidiu; Wang, Jiaping. \emph{  Conical structure for shrinking Ricci solitons}. J. Eur. Math. Soc. (JEMS) \textbf{19} (2017), no. 11, 3377--3390.

\bibitem[MW19]{MW19} 
Munteanu, Ovidiu; Wang, Jiaping.
\emph{Structure at infinity for shrinking Ricci solitons},
Annales Scientifiques de l'Ecole Normale Superieure \textbf{52} no.~4 (2019), 891--925.


\bibitem[NW08]{NW08} Ni, Lei; Wallach, Nolan. \emph{{On a classification of gradient shrinking solitons}}, Math. Res. Lett. \textbf{15} (2008), 941-–955.

\bibitem[Per02]{Per02} Perelman, Grisha. \emph{The entropy formula for the Ricci flow and its geometric applications}, arXiv:math.DG/0211159 (2002).


\bibitem[Per03]{Per03} Perelman, Grisha. \emph{Ricci flow with
surgery on three-manifolds}, arXiv:math.DG/0303109 (2003).

\bibitem[PW10]{PW10} Petersen, Peter; Wylie, William. \emph{{On the classification of gradient Ricci solitons}}, Geom. Topol. \textbf{14} (2010), 2277-–2300.

\bibitem[T21]{T21} 
Thurston, William P. 
\emph{The Geometry and Topology of Three-Manifolds: With a Preface by Steven P. Kerckhoff}, Collected Works IV,
Volume \textbf{27} (2021).

\bibitem[Wu13]{Wu13}
Wu, Peng. \emph{On the potential function of gradient steady Ricci solitons.} J. Geom. Anal. \textbf{23}
(2013), no. 1, 221--228.


 
\end{thebibliography}
\bibliographystyle{amsalpha}

\newcommand{\etalchar}[1]{$^{#1}$}
\providecommand{\bysame}{\leavevmode\hbox to3em{\hrulefill}\thinspace}
\providecommand{\MR}{\relax\ifhmode\unskip\space\fi MR }
\providecommand{\MRhref}[2]{%
  \href{http://www.ams.org/mathscinet-getitem?mr=#1}{#2}
}
\providecommand{\href}[2]{#2}

\end{document}